\documentclass[11pt]{amsart}
\usepackage{amsxtra}
\usepackage{amssymb}
\usepackage{mathrsfs}
\usepackage{enumerate}
\usepackage{array}
\theoremstyle{plain}

\newtheorem{theorem}{Theorem}[section]

\newtheorem{definition-theorem}[theorem]{Definition-Theorem}

\newtheorem{corollary}[theorem]{Corollary}

\title{Bispectrality for Matrix Laguerre-Sobolev polynomials}
\author{Francisco Marcellán}
\address[F. Marcellán]{Departamento de Matemáticas, Universidad Carlos III de Madrid,
Leganés, Spain.}
\email{pacomarc@ing.uc3m.es}
\author{Ignacio Zurrián}
\address[I. Zurrián]{Departamento de Matemática Aplicada II, Universidad de Sevilla, Seville, Spain.}
\email[Corresponding author]{ignacio.zurrian@fulbrightmail.org}

\begin{document}
\maketitle

\textbf{Abstract}\\

In this contribution we deal with sequences of polynomials orthogonal with respect to a Sobolev type inner product. A banded symmetric operator is associated with such a sequence of polynomials according to the higher order difference equation they satisfy. Taking into account the Darboux transformation of the corresponding matrix we deduce the connection with a sequence of orthogonal polynomials associated with a Christoffel perturbation of the measure involved in the standard part of the Sobolev inner product. A connection with matrix orthogonal polynomials is stated. The Laguerre-Sobolev type case is studied as an illustrative example. Finally, the bispectrality of such matrix orthogonal polynomials is pointed out.\\

\bigskip



\section{Introduction}

The study of inner products associated with a vector of measures $(d\mu_{0}, d\mu_{1}, \cdots, d\mu_{N})$ supported on the real line has attracted the interest of many researchers taking into account many properties of standard orthogonal polynomials are lost (see the survey paper \cite{MX15}).  In particular, the multiplication operator by $x$ is not symmetric with respect to such inner products and, as a consequence, the corresponding sequences of orthogonal polynomials do not satisfy a three term recurrence relation, that plays a central role in the theory of standard orthogonal polynomials (see \cite{Ch78}).  The matrix counterpart of the three term recurrence relation is a tridiagonal matrix that is known in the literature as Jacobi matrix. The spectral theory of such Jacobi matrices is an old topic and yields the so called Favard theorem (see \cite{Ch78}). Assuming you have $LU$ and $UL$ factorization, respectively, of a shifted Jacobi matrix, then the commutation between the matrices in the above factorizations yields new Jacobi matrices whose spectral resolution generates the canonical  Christoffel and Geronimus transformations, respectively (see \cite{BM04}, \cite{GMM21}, \cite{Ga02}, \cite{Ga04}  \cite{Y02}, \cite{Z97}, among others). They are the discrete counterpart of the Darboux transformations for second order linear differential operators. When you consider a canonical  Christoffel transformation and next a canonical Geronimus transformation of a Jacobi matrix, then the resulting Jacobi matrix has as spectral resolution the so called Uvarov transformation that is a perturbation of the initial spectral measure by adding a Dirac mass point. They appear in the framework of the spectral analysis of fourth order differential operators with polynomial coefficients as analyzed in the pioneering work \cite{HLK40}.\\

The implementation of multiple Christoffel transformations, i. e., an iteration of canonical Christoffel transformations,  has been studied in \cite{Ga04}. On the other hand,  in \cite{DGM14} the authors focus the attention on multiple Geronimus transformations in a more general framework.\\

When you deal with a Sobolev inner product associated with a vector of measures as above but $d\mu_{k}, k=1, 2, \cdots, N,$ are supported on finite subsets of the real line, the so called Sobolev-type inner product appears. The corresponding sequences of orthogonal polynomials are "no so far" of the sequences of standard orthogonal polynomials with respect to the measure $d\mu_{0}.$ This fact was pointed out in \cite{AMRR92} when  $N=1$  and the support of $d\mu_{1}$ is  a point $c$ in the real line that can also be a mass point of the measure $d\mu_{0}$ and in \cite{MR90} when $d\mu_{k}=0, k=1, 2, \cdots, N-1,$ and support of $d\mu_{N}$ is a point $c$ in the real line. Algebraic and analytic properties of such orthogonal polynomials have been extensively studied in the literature. In particular, when $d\mu_{0}$ is the gamma distribution  several authors have studied differential operators such that the corresponding eigenfunctions are orthogonal polynomials with respect to Sobolev type inner products assuming the support of the measures $d\mu_{k}, k=1, 2,\cdots, N$ is  $\{0\}.$ The pioneering work \cite{K90}  yields an intensive study about the existence and explicit expressions for such differential operators (see \cite{DI15}, \cite{KKB98}, \cite{KM93}, \cite{M19}). When $d\mu_{0}$ is the beta distribution, a similar analysis was done when the masses are located in one of the end points of the support, i. e., $\{\pm1\}$ (see \cite{DI18}, \cite{M21}, \cite{Ma21}, \cite{M22}).\\

Orthogonal polynomials with respect to Sobolev type  inner products satisfy higher order recurrence relations associated with a multiplication operator by a polynomial. Such an operator is symmetric with respect to the above inner product. The  converse result, an analogue of the Favard's theorem, has been studied in \cite{D93} where a representation of a general symmetric real bilinear form such that there exists a multiplication operator by a polynomial $x^{N+1}$ that is symmetric with respect to such a bilinear form $B$, i. e., the corresponding sequence of orthogonal polynomials satisfies a symmetric $2N+3$ recurrence relation, is given. Moreover, the following facts are equivalent (see Corollary 7 in \cite{D93}).

\begin{itemize}

\item  The multiplication operator by $x^{N+1}$ is a symmetric operator with respect to the bilinear form $B$, it commutes with the multiplication operator $x,$ i. e., if $p, q$ are polynomials,  then $B(x^{N+1} p(x), x q(x))= B (x p(x), x^{N+1}q(x)),$ and $B(x^{j}, x^{^k})= B(1, x^{j+k}), 1\leq j, k\leq N.$

\item  There exist a function $\mu_{0}$ and constants $M_{k}, 1\leq k\leq N,$ such that

$$B(p(x), q(x))= \int p(x) q(x) d\mu_{0}(x) + \sum_{k=1}^{N} M_{k} p^{(k)}(0) q^{(k)}(0).$$

\end{itemize}

In particular, it was shown in \cite{ELMMR95} that for a  Sobolev type  inner product
$$\langle f,g\rangle = \int f(x)\,g(x)\,d\mu(x)+\sum_{k=0}^N M_k f^{(k)}(c)g^{(k)}(c), \quad M_N>0,
$$
where $c$ is a point in $\mathbb R$, the multiplication by $(x-c)^{N+1},$ denoted by $E$,  is a symmetric operator and the sequence of orthogonal polynomials $\{s_n\}_{n\geq0}$ satisfies a $(2N+3)$-term recurrence relation of the form
$$(x-c)^{N+1}s_n(x)=\sum_{k=n-N-1}^{n+N+1} a_{n,k}\,s_k(x).$$
In other words, $s_n(x)$ is an eigenfunction of a linear difference operator $E$ in the variable $n$ with eigenvalue $(x-c)^{N+1}$. Notice that, according to \cite{ELMMR95}, if you have a Sobolev inner product
 $$
 \langle f, g\rangle =\sum_{k=0}^N  \int f^{(k)}(x)\,g^{(k)}(x)\,d\mu_{k}(x)
$$
and the multiplication by $(x-c)^{N+1}$ is a symmetric operator with respect to the above inner product, then $d \mu_{k}(x), k=1, 2, \cdots, N, $ are Dirac deltas supported at $x=c$ and the mass of $d \mu_{N}(x)$ is a positive real number. \\

On the other hand, in Theorem 6 \cite{D93} it  is proved  that the following  statements are equivalent.

\begin{itemize}

\item  The multiplication operator by $x^{N+1}$ is symmetric with respect to the bilinear form $B$ it commutes with the multiplication operator $x,$ i. e., $B(x^{N+1} p(x), x q(x))= B (x p(x), x^{N+1}q(x)),$ where $p, q$ are polynomials.

\item  There exist a function $\mu_{0}$ and a positive semi-definite matrix $M$ such that

$$B(p(x), q(x))= \int p(x) q(x) d\mu_{0}(x) + (p(0), p'(0), \cdots, p^{(N)}(0)) M  (q(0), q'(0), \cdots, q^{(N)}(0))^{t}.$$

\end{itemize}

This inner product is said to be a nondiagonal Sobolev type inner product. Zeros and asymptotic properties of sequences of orthogonal polynomials with respect to the above inner product have been studied in \cite{AMRR95}. A connection with bispectral problems when $d\mu_{0}(x)$ is the gamma distribution has been studied in \cite{DI20}.\\

The structure of the manuscript is as follows.  In Section \ref{Sec-DT}  we prove that a Darboux transformation of the operator $E$. i.e., $E=LU,$ gives rise to an operator $UL$ which has as eigenfunctions the orthogonal polynomials associated with $(x-c)^{N+1}d\mu(x)$. Furthermore, we prove that $UL$ actually is the $(N+1)$-th power of the {\it standard} three-term recurrence relation (TTRR in short) that the sequence of  polynomials orthogonal with respect to the measure $(x-c)^{N+1}d\mu(x)$ satisfies. Thus, we generalize a result given in \cite{HHLM22} when $N=1$ concerning the connection between the matrix representation, a five diagonal matrix in terms of the orthonormal basis $s_n(x)$,  of the multiplication operator by $(x-c)^2$ and the square of the shifted matrix $J_{2} -cI,$ where $J_{2}$ is the Jacobi matrix associated with the measure $(x-c)^2 d\mu(x).$\\

In Section \ref{dos} we set a matrix-valued approach by means of \cite{DvA95}. For this regard we consider the specific sequence of Laguerre-Sobolev type orthogonal polynomials to build a  monic matrix-valued orthogonal polynomial sequence  $\{P_n\}_{n\geq0}$ that satisfy a TTRR with matrix coefficients and we perform a Darboux transformation to find a very interesting connection with results in \cite{DS02}. Namely, we start with a matrix-valued TTRR
$$xP_n (x)=P_{n+1}(x)+(\zeta_{2n+1}+\zeta_{2n})P_n (x)+\zeta_{2n}\zeta_{2n-1}P_{n-1}(x), n\geq0, P_{-1}(x)=0, $$
which, after a Darboux transformation, yields a TTRR satisfied by another sequence of monic matrix orthogonal  polynomials $\{Q_n\}_{n\geq0}$
$$
xQ_n (x)=Q_{n+1}(x)+(\zeta_{2n+2 }+\zeta_{2n+1})Q_n(x)+\zeta_{2n+1}\zeta_{2n }Q_{n-1}(x), n\geq0, Q_{-1}(x)=0.
$$
Here, the coefficients $\zeta_n$ are such that $$xW_n (x)=W_{n+1}(x)+\zeta_n W_{n-1}(x), n\geq0, W_{-1}(x)=0,$$
where $\{W_n\}_{n\geq0}$ is a sequence of  monic matrix orthogonal polynomials given by $W_{2n}(x)=P_n(x^2),  n\geq0, $ and $W_{2n+1}(x)=x Q_n(x^2), n\geq0$.\\

Finally, in Section \ref{tres}, we consider the Laguerre-Sobolev type inner product with $\alpha\in\mathbb N$, $N=1$ and $M_1>0, M_0=0$, to construct a differential operator of order $8$ that has every $P_n$ as eigenfunction, showing an underlaying matrix-valued bispectrality. Lastly, we prove that any matrix-valued orthogonal polynomial built from bispectral scalar polynomials with the aid of \cite{DvA95} is bispectral too.  Furthermore, we give a general an explicit method to build the corresponding differential operator.

\section{Sobolev polynomials under Darboux transformation}\label{Sec-DT}

Given a probability measure $\mu$ supported on an infinite subset of the real line, a point $c$ in the real line and a positive integer $N$, we consider the following  inner products: the one mentioned in the introduction
\begin{equation}\label{inner-product}
\langle f,g\rangle = \int f(x)\,g(x)\,d\mu(x)+\sum_{j,k=0}^N M_{j,k} f^{(j)}(c)g^{(k)}(c),
\end{equation}
where $(M_{j,k})_{j,k=0}^{N}$ is a positive semi-definite matrix of size  $(N+1)\times(N+1)$, and
other one of the form
\begin{equation}\label{inner-product-2}
\langle f,g\rangle_{N+1} = \int f(x)\,g(x)\,(x-c)^{N+1}d\mu(x).
\end{equation}

Now, let us denote by  $\{s_n\}_{n\geq0}$ and $\{p_n\}_{n\geq0}$ the sequences of orthonormal polynomials with respect to \eqref{inner-product} and \eqref{inner-product-2}, respectively. Immediately, one realizes that, since \begin{align}
\langle s_n,p_j\rangle_{N+1}=\langle s_n,(x-c)^{N+1} p_j\rangle_{0}=\langle (x-c)^{N+1}s_n, p_j\rangle
\end{align}
is equal to $0$ for $j<n-N-1$, we have
$$s_n (x)=\sum_{j=n-N-1}^{n} T_{n,j}p_j (x),$$
for some coefficients $T_{n,j}$.

For any two sequences of polynomials $\{\alpha_j\}_{j\geq0}$ and $\{\beta_j\}_{j\geq0}$ one can consider the vector notation $\alpha=(\alpha_0,\alpha_1,\dots)^t$ and $\beta=(\beta_0,\beta_1,\dots)^t$. Furthermore, for any inner product $B(\cdot,\cdot)$ we can also consider the bilinear form $B( \alpha,\beta)$ which is nothing more than the semi-infinite matrix whose $(j,k)$-entry is given by $B( \alpha_j,\beta_k )$.
With this notation, if we call $s=(s_0,s_1,\dots)^t$, $p=(p_0,p_1,\dots)^t$,
 and we
define the semi-infinite nonsingular matrix $T= (T_{n,j})_{n,j=0}^{\infty}$,  then we have $s=Tp$ and therefore
$$
\langle s,s \rangle_{N+1}=\langle Tp,Tp \rangle_{N+1}=TT^*.
$$
Recall in this connection that, by definition, the matrix $T$ is not only lower triangular and nonsingular but also has zero entries below the $(N+1)$-th subdiagonal.

On the other hand, we have that the sequence of orthonormal polynomials $\{s_n\}_{n\geq0}$ satisfies a $(2N+3)$-term recurrence relation of the form
$$(x-c)^{N+1}s_n(x)=\sum_{k=n-N-1}^{n+N+1} h_{n,k}\,s_k(x).$$
This defines a matrix $H$ such that
\begin{equation}\label{H}(x-c)^{N+1}s=Hs.\end{equation}
Since
$$\langle s,s \rangle_{N+1}=\langle Hs,s \rangle=H,$$
we also have the following factorization of $H$ $$H=TT^*.$$
From \eqref{H} we now have that
\begin{equation}\label{DT}
(x-c)^{N+1}s=TT^*s\quad\text{ and } \quad (x-c)^{N+1}p=T^*Tp.
\end{equation}
This can be summarized as follows.

\begin{theorem}\label{a}
For any probability measure $\mu$ supported on an infinite subset of the real line, a point $c$ in the real line and a positive integer $N$, the  sequence of Sobolev-type orthonormal polynomials $\{s_n\}_{n\geq0}$ with respect to
$$
\langle f,g\rangle = \int f(x)\,g(x)\,d\mu(x)+\sum_{j,k=0}^N M_{j,k} f^{(j)}(c)g^{(k)}(c),
$$
is a  Darboux transformation of the sequence of orthonormal polynomials $\{p_n\}_{n\geq0}$ with respect to $$
\langle f,g\rangle_{N+1} = \int f(x)\,g(x)\,(x-c)^{N+1}d\mu(x),
$$
 by means of \eqref{DT}. Namely, if we consider the TTRR  satisfied by the sequence of orthonormal polynomials $\{p_n(x)\}_{n\geq0},$ in vector notation $xp=J_{N+1}p$, the  symmetric matrix $(J_{N+1}-c)^{N+1}$ can be factorized  as $T^*T,$ where $s=Tp.$ Notice that the matrix $T= (T_{n,j})_{n,j=0}^{\infty}$  can be calculated explicitly
$$T_{n,j}  =
\left\{
\begin{aligned}
      & \left \langle s_n,p_j\right\rangle_{N+1}&   &n-N-1\le j\le n, \\ 
      &  0&  \quad &\text{elsewhere}.
\end{aligned}
\right.
$$
\end{theorem}

\

As a  straightforward consequence of the above theorem, when in \eqref{inner-product} $M_{j,k}=0, j, k= 0,1, \cdots, N,$ we get
\begin{corollary}
For any probability measure $\mu$ supported on an infinite subset of the real line, a point $c$ in the real line and a positive integer $N$, the sequence of orthonormal polynomials $\{q_n\}_{n\geq0}$ with respect to
$$
\langle f,g\rangle_{0} = \int f(x)\,g(x)\,d\mu(x)
$$
is  a  Darboux transformation of the sequence of orthonormal polynomials $\{p_n\}_{n\geq0}$ with respect to
$$
\langle f,g\rangle_{N+1} = \int f(x)\,g(x)\,(x-c)^{N+1}d\mu(x).$$
Namely, if we consider the TTRR satisfied by $\{p_n\}_{n\geq0}$ in vector notation $x p=J_{N+1} p$, the  symmetric matrix  $(J_{N+1}-c)^{N+1}$ can be factorized as $C^* C $ with $p=Cq.$ Furthermore,
$$C_{n,k}=
\left\{
\begin{aligned}
 & \left\langle q_n,p_k\right\rangle_{{N+1}}&   &n-N-1\le k\le n, \\ 
 & 0&  \quad &\text{elsewhere}.
\end{aligned}
\right.
$$
\end{corollary}

\bigskip

Finally, let us observe that from the TTRR satisfied by $\{p_n\}_{n\geq0}$ in vector notation $xp=J_{N+1}p$ we have
$$
(x-c)^{N+1}p=(J_{N+1}-c)^{N+1} p.
$$
From what we saw above, $(J_{N+1}-c)^{N+1}$ admits an $UL$-factorization with $U=L^*$, which of course is not unique .
We conjecture that all such factorizations give rise to one of the families already considered above.

%

\section{Matrix-valued orthogonal polynomials}\label{dos}

In this section we will restrict ourselves to the particular case when $c=0$. Furthermore, for reasons of space we will simplify the notation by considering $\alpha=0$, $N=1$ and the inner product
\begin{equation}\label{inner}
\langle f,g\rangle=\int_0^\infty f(x)g(x)e^{-x}dx+f'(0)g'(0).
\end{equation}
The interested reader can verify that the results in the present section hold for more general $\alpha$ and $N.$ Nevertheless we believe that a $2\times2$ matrix-valued construction with $\alpha=0$ will suffice to illustrate the situation.

Let us denote   by $\{\mathcal L_n\}_{n\geq0}$  the sequence of orthonormal polynomials with respect to the inner product \eqref{inner}. Thus
we have
$$x^2 \mathcal L_n (x)=a_n\mathcal L_{n+2}(x)+b_n\mathcal L_{n+1}(x)+c_n\mathcal L_{n}(x)+b_{n-1}\mathcal L_{n-1}(x)+a_{n-2}\mathcal L_{n-2}(x),\quad n\geq 2,$$
with

\begin{align*}
a_n&=\sqrt{\frac{{\left(2 \, n^{2} + 7 \, n + 9\right)} {\left(2 \, n^{2} - 5 \, n + 6\right)} {\left(n + 4\right)} {\left(n + 2\right)} {\left(n + 1\right)}^{3}}{{\left(2 \, n^{2} + 3 \, n + 4\right)} {\left(2 \, n^{2} - n + 3\right)} {\left(n + 3\right)}}},
\\
b_n&=4 \, \sqrt{\frac{{\left(4 \, n^{7} + 16 \, n^{6} + 13 \, n^{5} + 10 \, n^{4} + 43 \, n^{3} + 64 \, n^{2} + 84 \, n + 36\right)}^{2} {\left(n + 1\right)}}{{\left(2 \, n^{2} + 3 \, n + 4\right)} {\left(2 \, n^{2} - n + 3\right)}^{2} {\left(2 \, n^{2} - 5 \, n + 6\right)} {\left(n + 3\right)} {\left(n + 2\right)}^{2}}},
\\
c_n&=
2 \, \sqrt{\frac{{\left(12 \, n^{8} + 12 \, n^{7} - 23 \, n^{6} + 57 \, n^{5} + 82 \, n^{4} - 81 \, n^{3} + 37 \, n^{2} + 120 \, n + 36\right)}^{2}}{{\left(2 \, n^{2} - n + 3\right)}^{2} {\left(2 \, n^{2} - 5 \, n + 6\right)}^{2} {\left(n + 2\right)}^{2} {\left(n + 1\right)}^{2}}}.
\end{align*}

Let $\{R_{0,n}\}_{n\geq0},\{R_{1,n}\}_{n\geq0}$  be the sequences of polynomials such that for any $n$  $$\mathcal L_n(x)= x R_{1,n}(x^2)+R_{0,n}(x^2).$$
Then, following \cite{DvA95}, we build the matrix-valued polynomials
\begin{equation}\label{R_n}
R_n(y)=
\left(\begin{matrix}
R_{0,2n}(y)&R_{1,2n}(y)
\\
R_{0,2n+1}(y)&R_{1,2n+1}(y)
\end{matrix}\right).
\end{equation}

The sequence $\{R_n\}_{n\geq0}$ satisfies a matrix TTRR
\begin{equation}\label{n3t}
xR_n(y)=A_{n-1}^*R_{n-1}(y)+B_nR_{n}(y)+ A_nR_{n+1}(y), n\geq0,
\end{equation}
with $A_n, B_n$ given, respectively, by

{\tiny
\begin{align*}
{A_n}_{0,0}=&\frac{2 \, \sqrt{8 \, n^{2} + 14 \, n + 9} \sqrt{4 \, n^{2} - 5 \, n + 3} {\left(2 \, n + 1\right)}^{\frac{3}{2}} \sqrt{n + 2} \sqrt{n + 1}}{\sqrt{8 \, n^{2} - 2 \, n + 3} \sqrt{4 \, n^{2} + 3 \, n + 2} \sqrt{2 \, n + 3}},
\\
{A_n}_{0,1}=&0,
\\
{A_n}_{1,0}=&
-\frac{4 \, {\left(256 \, n^{7} + 1408 \, n^{6} + 3088 \, n^{5} + 3640 \, n^{4} + 2692 \, n^{3} + 1414 \, n^{2} + 570 \, n + 135\right)} \sqrt{n + 1}}{\sqrt{8 \, n^{2} + 14 \, n + 9} \sqrt{8 \, n^{2} - 2 \, n + 3} {\left(4 \, n^{2} + 3 \, n + 2\right)} {\left(2 \, n + 3\right)} \sqrt{n + 2}}
,
\\
{A_n}_{1,1}=&\frac{2 \, \sqrt{8 \, n^{2} - 2 \, n + 3} \sqrt{4 \, n^{2} + 11 \, n + 9} \sqrt{2 \, n + 5} \sqrt{2 \, n + 3} {\left(n + 1\right)}^{\frac{3}{2}}}{\sqrt{8 \, n^{2} + 14 \, n + 9} \sqrt{4 \, n^{2} + 3 \, n + 2} \sqrt{n + 2}}
,
\\
{B_n}_{0,0}=&
\frac{2 \, {\left(768 \, n^{8} + 384 \, n^{7} - 368 \, n^{6} + 456 \, n^{5} + 328 \, n^{4} - 162 \, n^{3} + 37 \, n^{2} + 60 \, n + 9\right)}}{{\left(8 \, n^{2} - 2 \, n + 3\right)} {\left(4 \, n^{2} - 5 \, n + 3\right)} {\left(2 \, n + 1\right)} {\left(n + 1\right)}}
,
\\
{B_n}_{0,1}=&-\frac{4 \, {\left(128 \, n^{7} + 256 \, n^{6} + 104 \, n^{5} + 40 \, n^{4} + 86 \, n^{3} + 64 \, n^{2} + 42 \, n + 9\right)} \sqrt{2 \, n + 1}}{{\left(8 \, n^{2} - 2 \, n + 3\right)} \sqrt{4 \, n^{2} + 3 \, n + 2} \sqrt{4 \, n^{2} - 5 \, n + 3} \sqrt{2 \, n + 3} {\left(n + 1\right)}}
,
\\
{B_n}_{1,0}=&
-\frac{4 \, {\left(128 \, n^{7} + 256 \, n^{6} + 104 \, n^{5} + 40 \, n^{4} + 86 \, n^{3} + 64 \, n^{2} + 42 \, n + 9\right)} \sqrt{2 \, n + 1}}{{\left(8 \, n^{2} - 2 \, n + 3\right)} \sqrt{4 \, n^{2} + 3 \, n + 2} \sqrt{4 \, n^{2} - 5 \, n + 3} \sqrt{2 \, n + 3} {\left(n + 1\right)}}
,
\\
{B_n}_{1,1}=&\frac{2 \, {\left(768 \, n^{8} + 3456 \, n^{7} + 6352 \, n^{6} + 6744 \, n^{5} + 5128 \, n^{4} + 2898 \, n^{3} + 1099 \, n^{2} + 303 \, n + 63\right)}}{{\left(8 \, n^{2} - 2 \, n + 3\right)} {\left(4 \, n^{2} + 3 \, n + 2\right)} {\left(2 \, n + 3\right)} {\left(n + 1\right)}}
.
\end{align*}
}

Thus $\{R_n\}_{n\geq0}$ is a  sequence of matrix orthonormal polynomials  with respect to the positive semi-definite matrix-valued inner product given by
$$\langle F,G\rangle=\int_0^\infty F(y)\left(\begin{matrix}	1&\sqrt y\\ \sqrt y & y
\end{matrix}\right)G^*(y)e^{-y}dy+F(0)\left(\begin{matrix}	0&0\\ 0 & 1
\end{matrix}\right)G^*(0),$$
for $2\times2$ matrix-valued functions $F,G$.

We now explore the Darboux process for  the above matrix TTRR, but applied to the monic matrix orthogonal polynomials. Since the leading coefficient of $R_n$ is given by
$$
\left(\begin{smallmatrix}
\frac{\sqrt{4 \, n^{2} - 5 \, n + 3}  \sqrt{2 \, n + 1}}{4 \, \sqrt{8 \, n^{2} - 2 \, n + 3} {\left(2 \, n - 1\right)} \sqrt{n + 1} {\left(n - 1\right)} n \left(2 \, n - 3\right)!} & 0 \\
\frac{{\left(8 \, n^{3} + 6 \, n^{2} - 5 \, n + 3\right)}  {\left(2 \, n + 1\right)}}{4 \, \sqrt{8 \, n^{2} - 2 \, n + 3} \sqrt{4 \, n^{2} + 3 \, n + 2} \sqrt{2 \, n + 3} {\left(2 \, n - 1\right)} {\left(2 \, n - 3\right)} \sqrt{n + 1} {\left(n - 1\right)} n \left(2 \, n - 4\right)!} &
\frac{ -\sqrt{8 \, n^{2} - 2 \, n + 3} \sqrt{n + 1}}
{\sqrt{4 \, n^{2} + 3 \, n + 2} \sqrt{2 \, n + 3} \left(2 n\right)!}
\end{smallmatrix}\right), $$
we can build explicitly the sequence of monic matrix orthogonal polynomials $\{P_n\}_{n\geq0}$. They  will satisfy a matrix TTRR such that the corresponding  Jacobi matrix of $(2\times2)$-blocks can be decomposed in the form $LU$  where $L$ is a lower block triangular matrix and $U$ is a block upper triangular  matrix. From Theorem \ref{a} the Darboux transformation will give rise to a sequence of monic matrix polynomials $\{Q_n\}_{n\geq0}$ orthogonal with respect to the  weight $e^{-y}$ multiplied by $y=x^2$. More precisely, the sequences $\{P_n\}_{n\geq0}$ and $\{Q_n\}_{n\geq0}$ satisfy
\begin{equation}\label{eq}
\begin{aligned}
	xP_n(x)=&P_{n+1}(x)+(\zeta_{2n+1}+\zeta_{2n})P_n(x)+\zeta_{2n}\zeta_{2n-1}P_{n-1}(x), n\geq 0,
\\
	xQ_n (x)=&Q_{n+1}(x)+(\zeta_{2n+2 }+\zeta_{2n+1})Q_n(x)+\zeta_{2n+1}\zeta_{2n }Q_{n-1}(x), n\geq 0,
\end{aligned}
\end{equation}
where
\begin{align*}
\zeta_{2n}=&\left(\begin{array}{rr}
-\frac{2 \, {\left(16 \, n^{2} - 12 \, n - 9\right)} {\left(2 \, n - 1\right)}^{2} {\left(n - 1\right)} n}{{\left(4 \, n^{2} - 5 \, n + 3\right)} {\left(2 \, n + 1\right)}} & \frac{4 \, {\left(8 \, n^{3} - 12 \, n^{2} + 4 \, n + 3\right)} n}{{\left(4 \, n^{2} - 5 \, n + 3\right)} {\left(2 \, n + 1\right)}} \\
-\frac{2 \, {\left(16 \, n^{3} - 40 \, n^{2} + 28 \, n - 3\right)} {\left(2 \, n + 1\right)} {\left(2 \, n - 1\right)}^{2} n}{4 \, n^{2} - 5 \, n + 3} & \frac{2 \, {\left(16 \, n^{3} - 36 \, n^{2} + 29 \, n - 6\right)} {\left(2 \, n + 1\right)} n}{4 \, n^{2} - 5 \, n + 3}
\end{array}\right),
\\
\zeta_{2n-1}=&\left(\begin{array}{rr}
-\frac{2 \, {\left(32 \, n^{4} + 8 \, n^{3} - 14 \, n^{2} + 7 \, n + 3\right)} {\left(2 \, n - 1\right)} n}{{\left(4 \, n^{2} - 5 \, n + 3\right)} {\left(2 \, n + 1\right)}} & \frac{4 \, {\left(8 \, n^{3} - 2 \, n + 3\right)} n}{{\left(4 \, n^{2} - 5 \, n + 3\right)} {\left(2 \, n + 1\right)}} \\
-\frac{2 \, {\left(32 \, n^{4} + 16 \, n^{3} - 32 \, n^{2} + 14 \, n + 9\right)} {\left(2 \, n + 1\right)} {\left(2 \, n - 1\right)} n}{4 \, n^{2} - 5 \, n + 3} & \frac{2 \, {\left(16 \, n^{3} + 4 \, n^{2} - 15 \, n + 12\right)} {\left(2 \, n + 1\right)} n}{4 \, n^{2} - 5 \, n + 3}
\end{array}\right).
\end{align*}
These TTRR are related through a Darboux transformation. Namely, we have
\begin{align*}
	x
\left(\begin{matrix}
P_0\\
P_1\\
P_2\\
P_4\\
\vdots
\end{matrix}\right)
=&
\left(\begin{matrix}
1&0&0&0&\cdots\\
\zeta_2&1&0&0&\cdots\\
0&\zeta_4&1&0&\cdots\\
0&0&\zeta_6&1&\ddots\\
\vdots & \vdots&\ddots&\ddots&\ddots
\end{matrix}\right)
\times
\left(\begin{matrix}
\zeta_1&1&0&0&\cdots\\
0&\zeta_3&1&0&\cdots\\
0&0&\zeta_5&1&\ddots\\
0&0&0&\zeta_7&\ddots\\
\vdots & \vdots&\ddots&\ddots&\ddots
\end{matrix}\right)
\left(\begin{matrix}
P_0\\
P_1\\
P_2\\
P_4\\
\vdots
\end{matrix}\right),
\\
		x
\left(\begin{matrix}
Q_0\\
Q_1\\
Q_2\\
Q_4\\
\vdots
\end{matrix}\right)
=&
\left(\begin{matrix}
\zeta_1&1&0&0&\cdots\\
0&\zeta_3&1&0&\cdots\\
0&0&\zeta_5&1&\ddots\\
0&0&0&\zeta_7&\ddots\\
\vdots & \vdots&\ddots&\ddots&\ddots
\end{matrix}\right)
\times
\left(\begin{matrix}
1&0&0&0&\cdots\\
\zeta_2&1&0&0&\cdots\\
0&\zeta_4&1&0&\cdots\\
0&0&\zeta_6&1&\ddots\\
\vdots & \vdots&\ddots&\ddots&\ddots
\end{matrix}\right)
\left(\begin{matrix}
Q_0\\
Q_1\\
Q_2\\
Q_4\\
\vdots
\end{matrix}\right).
\end{align*}

Equation \eqref{eq} may be compared with \cite[Lemma 3.3]{DS02}.

It is worth to notice that the coefficients for the TTRR of $\{Q_n\}_{n\geq0}$ are nicer. Indeed,
$$
\zeta_{2n+2 }+\zeta_{2n+1}=
4(n+1)
\left(\begin{array}{rr}
-{\left(4 \, n + 3\right)} {\left(2 \, n + 1\right)} & 2 \\
-2 \, {\left(4 \, n^{2} + 8 \, n + 5\right)} {\left(2 \, n + 3\right)} {\left(2 \, n + 1\right)} & {\left(4 \, n + 5\right)} {\left(2 \, n + 3\right)}
\end{array}\right),
$$
$$
\zeta_{2n+1}\zeta_{2n }=
4(n+1)n(2n+1)
\left(\begin{array}{rr}
-{\left(8 \, n + 3\right)} {\left(2 \, n - 1\right)} & 4 \\
-4 \, {\left(2 \, n + 3\right)} {\left(2 \, n + 1\right)}^{2} {\left(2 \, n - 1\right)} & {\left(8 \, n + 5\right)} {\left(2 \, n + 3\right)}
\end{array}\right).
$$
This is in concordance with the results of Section \ref{Sec-DT}. Indeed,  the sequence of monic matrix polynomials $\{Q_n\}_{n\geq0}$ is built from a sequence of polynomials satisfying a five term recurrence relation that, as we  proved above, is the square iterated of a {\it standard} TTRR and, as a consequence, we get a very simple expression for their coefficients.

\section{Matrix-valued Bispectrality}\label{tres}

Once one realizes that the sequence of matrix-valued polynomials $\{R_n\}_{n\geq0}$ given by \eqref{R_n} is related via a Darboux transformation with classical standard orthogonal polynomials it is natural to seek for matrix linear differential equations. As in the previous section we will set $c=0$, $\alpha=0$, $N=1$ and the inner product \eqref{inner}, for an initial examplification; at the end we will deal with arbitrary size $N+1$ and general coefficients.

After  straightforward computations one can see that for the following $2\times2$ matrix-valued operator,
$$
\mathfrak{D}=  \sum_{k=0}^8\frac{{d}^k}{{d}x^k} D_k(x),
$$
acting on the right-hand side of the polynomial $R_{n}(x)$, yields
$$R_n (x) \mathfrak{D}=\Lambda_nR_n (x),$$
where
\begin{align*}
D_0=&\left(\begin{array}{rr}
0 & 0 \\
-3 & 3
\end{array}\right),\quad
D_1= \left(\begin{array}{rr}
9 \, y - 6 & -12 \\
-105 \, y + 54 & 24 \, y
\end{array}\right),\\
D_2=&\left(\begin{array}{rr}
27 \, y^{2} + 474 \, y - 72 & -276 \, y \\
-6 \, {\left(151 \, y + 459\right)} y & 3 \, {\left(19 \, y + 1100\right)} y
\end{array}\right),
\\
D_3=&\left(\begin{array}{rr}
24 \, {\left(y^{2} + 166 \, y + 93\right)} y & -12 \, {\left(53 \, y + 570\right)} y \\
-4 \, {\left(287 \, y + 6852\right)} y^{2} & 8 \, {\left(4 \, y^{2} + 1278 \, y + 2205\right)} y
\end{array}\right),
\\
D_4=& \left(\begin{array}{rr}
4 \, {\left(y^{2} + 1080 \, y + 4701\right)} y^{2} & -8 \, {\left(37 \, y + 2253\right)} y^{2} \\
-8 \, {\left(47 \, y + 4908\right)} y^{3} & 4 \, {\left(y^{2} + 1770 \, y + 14301\right)} y^{2}
\end{array}\right),
\\
D_5=&\left(\begin{array}{rr}
96 \, {\left(13 \, y + 252\right)} y^{3} & -32 \, {\left(y + 348\right)} y^{3} \\
-32 \, {\left(y + 534\right)} y^{4} & 384 \, {\left(4 \, y + 123\right)} y^{3}
\end{array}\right),\\
D_6=& \left(\begin{array}{rr}
96 \, {\left(y + 101\right)} y^{4} & -2208 \, y^{4} \\
-2656 \, y^{5} & 32 \, {\left(3 \, y + 443\right)} y^{4}
\end{array}\right)
,\\
D_7=&\left(\begin{array}{rr}
1408 \, y^{5} & -128 \, y^{5} \\
-128 \, y^{6} & 1664 \, y^{5}
\end{array}\right),\quad
D_8=\left(\begin{array}{rr}
64 \, y^{6} & 0 \\
0 & 64 \, y^{6}
\end{array}\right)
\end{align*}
and
$$
\Lambda_n=
\left(\begin{array}{rr}
 {\left(4 \, n^{3} - n + 6\right)} n & 0 \\
0 &   {\left(2 \, n^{3} + 3 \, n^{2} + n + 3\right)} {\left(2 \, n + 1\right)}
\end{array}\right).
$$
Furthermore, it is easy to check that there is not a linear differential operator of order less than $8$ having $\{R_n\}_{n\geq0}$ as eigenfunctions.
On the other hand, the results in \cite{KKB98} prove the existence of a linear differential operator of order $2\alpha+8$ (see \cite[Theorem 3.1]{KKB98}).
This, of course, is not a coincidence as we will  show below.

\subsection{Bispectrality for general size}

Following the construction in \cite[page 265]{DvA95}, let $N\in \mathbb N$ and $\{s_n\}_{n\geq0}$ be  a sequence of orthonormal polynomials,  satisfying the $(2N+3)$-term recurrence relation
$$
x^{N+1}s_n(x)=\sum_{k=0}^{N+1}\left( c_{n,k}s_{n-k}(x)+\overline{c_{n+k,k}}s_{n+k}(x)  \right), n\geq0,
$$
where $c_{n,k}$ are complex numbers, $c_{n,N}=0$ for any $n$, the degree of $s_n$ is $n$, and $s_n=0$ for $n<0$.

For any $n,$ let $R_{k,n}$, $k=0,\dots,n$, be polynomials such that
$$
s_n(x)=R_{0,n}(x^{N+1})+x R_{1,n}(x^{N+1})+\dots+x^N R_{N,n}(x^{N+1}).
$$

Let

\begin{equation}\label{R_ngen}
R_n(y)=
\left(\begin{matrix}
R_{0,(N+1)n}(y)   &    R_{1,(N+1)n}(y) & \cdots & R_{N,(N+1)n}(y)
\\
R_{0,(N+1)n+1}(y)   &    R_{1,(N+1)n+1}(y) & \cdots & R_{N,(N+1)n+1}(y)
\\
\vdots                        &                    \vdots    &             & \vdots
\\
R_{0,(N+1)n+N}(y)   &    R_{1,(N+1)n+N}(y) & \cdots & R_{N,(N+1)n+N}(y)
\end{matrix}\right).
\end{equation}
In \cite{DvA95} it is proved that $\{R_n\}_{n\geq0}$ is a sequence of matrix polynomials that satisfies a TTRR. We will prove that if there exists a differential operator $ {D}$ having every $s_n(x)$ as eigenfunction, then there is a matrix-valued differential operator that has every $R_n(y)$ as eigenfunction.

Before stating the theorem let us introduce some notation. We denote by  $w$ the $(N+1)$-th root of unity $e^{i\tfrac{2\pi}{N+1}}$. For $y\neq 0$ we denote by ${|y|}^{\tfrac {j-1}{N+1}}$ its only positive $(N+1)$-th root. Then all the $(N+1)$-th roots  of $y$ are  $x=w^j{|y|}^{\tfrac {j-1}{N+1}}$ for $j=0,1,\dots, N$. Given a  differential operator ${D}$ with coefficients in the variable $x$ we will denote it by $ {D}(x)$ to emphasize the role of variable and by $ {D}\left(w^j{|y|}^{\tfrac {j-1}{N+1}}\right)$ the operator obtained after the change of variables $x\to w^j{|y|}^{\tfrac {j-1}{N+1}}$.

 \begin{theorem}
Let us assume that there exists a (scalar-valued) linear differential operator $ {D}(x)$ with polynomial coefficients  such that
$$ {D}(x) \, s_n(x) =\lambda_n\, s_n(x),\quad n=0,1,\dots.
$$
We consider the matrix-valued differential operator $\mathfrak D(y)$ acting on the right-hand side, given by  $$
\mathfrak D(y)=A(y)\,B\, C(y)\,B^{-1}\,A(y)^{-1},
$$
where 
$A(y)$ is a diagonal matrix, $B$ is a constant matrix and $C(y)$ is a (diagonal) matrix-valued operator acting on the right-hand side, all of size $(N+1)\times(N+1),$ such that
$$
A(y)_{j,j}={|y|}^{\tfrac {j-1}{N+1}},\quad
B_{j,k}=w^{(j-1)(k-1)},\quad
C(y)_{j,j} =  D\left(w^{j-1}{|y|}^{\tfrac {1}{N+1}}\right).
$$
Then
$$R_n(y)\,\mathfrak D(y)=\Lambda_nR_n(y),$$
where $R_n$ are the $(N+1)\times(N+1)$ matrix-valued polinomials given in \eqref{R_ngen} and $\Lambda_n$ is the diagonal eigenvalue matrix
$$\Lambda_n
=
\left(\begin{smallmatrix}
\lambda_{(N+1)n}   &     &  &
\\
  &   \lambda_{(N+1)(n+1)} &  &
\\
                        &                        &     \ddots        &
\\   &     &  & \lambda_{(N+1)N}
\end{smallmatrix}\right).
$$
\end{theorem}
\begin{proof}
By looking at the entry $(R_n(y)\,A(y)B)_{j,k}$  we have
$$
\left(\begin{matrix}
R_{0,(N+1)n+j-1}(y) & R_{1,(N+1)n+j-1} (y)& \dots &R_{N,(N+1)n+j-1} (y)
\end{matrix}\right)
\times
\left(\begin{matrix}
{|y|}^{\tfrac {0}{N+1}}      w^{0(k-1)} \\
{|y|}^{\tfrac {1}{N+1}}    w^{1(k-1)}\\
  \dots \\
{|y|}^{\tfrac {N}{N+1}}    w^{N(k-1)}
\end{matrix}\right)
$$
which is
\begin{multline*}
R_{0,(N+1)n+j-1}(y) {|y|}^{\tfrac {0}{N+1}}      w^{0(k-1)} +   \dots + R_{N,(N+1)n+j-1} (y){|y|}^{\tfrac {N}{N+1}}    w^{N(k-1)}
\\
=
s_{(N+1)n+j-1}\left({|y|}^{\tfrac {1}{N+1}}    w^{(k-1)}\right).
\end{multline*}

Since  $(\left(R_n(y)\,A(y)B\right)\, C(y))_{j,k}= D(w^{k-1}{|y|}^{\tfrac {1}{N+1}})\left((R_n(y)\,A(y)B)_{j,k}\right),$  we have
\begin{align*}
(R_n(y)\,A(y)B\, C(y))_{j,k}= &D(w^{k-1}{|y|}^{\tfrac {1}{N+1}})\left(s_{N,(N+1)n+j-1}\left({|y|}^{\tfrac {1}{N+1}}    w^{(k-1)}\right)\right)
\\
=&
\lambda_{(N+1)n+j-1} s_{(N+1)n+j-1}\left({|y|}^{\tfrac {1}{N+1}}    w^{(k-1)}\right)
\\
=&\lambda_{(N+1)n+j-1} (R_n(y)\,A(y)B)_{j,k}.
\end{align*}
This implies that $R_n(y)\,A(y)B C(y)=\Lambda_nR_n(y)\,A(y)B.$

Thus
$$
R_n(y)\,A(y)B\, C(y)\,B^{-1}\,A(y)^{-1}=\Lambda_nR_n(y),
$$
and we get the desired statement.
\end{proof}

Notice that bispectrality for Krall-Laguerre orthogonal polynomials, an example of standard orthogonal polynomials,  when $\alpha$ is a positive integer, has been studied in \cite{DI20} by using a different approach.

\section*{Acknowledgements}

The research of F.  Marcell\'an has been supported by FEDER/Ministerio de Ciencia e Innovación-Agencia Estatal de Investigación of Spain, grant PID2021-122154NB-I00, and the Madrid
Government (Comunidad de Madrid-Spain) under the Multiannual Agreement with UC3M in the line of Excellence of University Professors, grant EPUC3M23 in the context of the V PRICIT  (Regional Program of Research and Technological Innovation).
The research of the author Ignacio Zurri\'an was partially supported by CONICET and by VI PPIT-US.

\bibliographystyle{alpha}


\bigskip 
\end{document}